\documentclass[11pt, reqno, english]{amsart}  
\usepackage[utf8]{inputenc}
\usepackage[T1]{fontenc}
\usepackage{amsmath,amsthm}
\usepackage{amsfonts,amssymb}
\usepackage{mathtools}
\usepackage[colorlinks=true,urlcolor=blue,linkcolor=red,citecolor=magenta]{hyperref}
\usepackage{enumerate, paralist}
\usepackage{tikz-cd}
\usepackage[left=1.25in,right=1.25in,top=1.2in,bottom=1.2in]{geometry}
\usepackage{xurl}

\usepackage{xstring,ifthen,tabularx}
\usepackage{tikz}

\newcounter{commentcounter}

\theoremstyle{plain}
\newtheorem{theorem}{Theorem}[section]
\newtheorem{lemma}[theorem]{Lemma}

\theoremstyle{definition}

\newtheorem{example}[theorem]{Example}
\newtheorem{remark}[theorem]{Remark}

\newtheorem{question}[theorem]{Question}

\newcommand{\R}{\mathbb{R}}

\newcommand{\N}{\mathbb{N}}

\thanks{FF was supported by NSF grant DMS 1855591, NSF CAREER grant DMS 2042428, and a Sloan Research Fellowship.  MH was supported by NSF grant DMS 1926686.  WP was supported by NSF grant DMS 2054503.}

\title[Unit sphere fibrations]{Unit sphere fibrations in Euclidean space}

\begin{document}



\author{Daniel Asimov}
\address[DA]{}
\email{asimov@msri.org}

\author{Florian Frick}
\address[FF]{Dept.\ Math.\ Sciences, Carnegie Mellon University, Pittsburgh, PA 15213, USA}
\email{frick@cmu.edu} 

\author{Michael Harrison}
\address[MH]{Institute for Advanced Study, Princeton, NJ 08540, USA}
\email{mah5044@gmail.com} 

\author{Wesley Pegden}
\address[WP]{Dept.\ Math.\ Sciences, Carnegie Mellon University, Pittsburgh, PA 15213, USA}
\email{wes@math.cmu.edu}


\maketitle
\date{\today}

\begin{abstract}
We show that if an open set in $\R^d$ can be fibered by unit $n$-spheres, then $d \geq 2n+1$, and if $d = 2n+1$, then the spheres must be pairwise linked, and $n \in \left\{ 0, 1, 3, 7 \right\}$.   For these values of $n$, we construct unit $n$-sphere fibrations in $\R^{2n+1}$.
\end{abstract}

\section{Motivation and statement of results}

A classical theorem states that $\R^3$ can be decomposed as a union of disjoint unit circles.  This statement may be folklore, but it seems to have been popularized by Conway and Croft, who gave a nonconstructive proof using transfinite induction \cite{Conway}.   Since then, a number of fascinating questions, regarding the possibility of decomposing a fixed space as the union of disjoint copies of another space, have been studied from set theoretic, topological, and geometric perspectives.   A collection of decomposition questions appear in an expository article by Martin Gardner \cite{Gardner}, whereas a collection of geometric constructions using transfinite induction can be found in \cite{Komjath}.   Similar questions appear every so often on MathOverflow; see \cite{of4,of3,of2,of1}.  It seems to be unknown whether there exists a Borel decomposition of $\R^3$ into unit circles; see \cite{of5}. 

With different topological or geometric constraints, versions of this question have appeared in a wide range of  articles \cite{BankstonMcGovern,Bass, Bing4,Bing5,Bing2,Bing3,Bing1,Cobb,JonssonWastlund,Moussong,Wilker}, and several explicit constructions have appeared.   For example, Szulkin constructed an explicit covering of $\R^3$ using geometric (round) circles of varying radii, which fails to be a foliation only on the union of a countable discrete set of concentric circles \cite{Szulkin}, whereas foliations of $\R^3$ by topological circles were thoroughly studied by Vogt \cite{Vogt89}.  Bankston and Fox \cite{Bankston}, generalizing a construction of Kakutani, gave a decomposition of $\R^{n+2}$ by tamely embedded, unknotted, pairwise unlinked (topological) $n$-spheres, but it seems unknown whether there exists any decomposition of $\R^{n+2}$ by unit $n$-spheres for $n > 1$.

Here we address the existence question for \emph{continuous} unit sphere decompositions of open subsets $E$ in Euclidean space $\R^d$.  A decomposition of $E$ by unit $n$-spheres is called continuous if the map taking $p \in E$ to its containing sphere is continuous; that is, if the sphere centers and their normal spaces vary continuously with $p$.   In this case we say that $E$ is fibered by unit $n$-spheres.  Global fibrations of $\R^d$ by unit spheres cannot exist for topological reasons, as can be seen from the long exact sequence associated to the fibration.  However, the existence question for local fibrations, has, to the authors' surprise, turned out to be nontrivial.

A simple example of a unit circle fibration can be described as follows.  Consider an open torus $T \in \R^3$ with major and minor radii both equal to $1$ (the boundary of $T$ has a singularity in the center).  The interior of this torus may be foliated by a collection of tori $T_r$, each with major radius $1$, but with minor radii $r$ ranging from $0$ to $1$.   Now each individual torus admits two foliations (a left-handed one and a right-handed one) by \emph{Villarceau circles}.  The radius of a Villarceau circle is equal to the major radius of the torus, and hence all such circles have radius $1$.  Choosing the right-handed foliation on each torus yields a collection of pairwise linked unit circles which fill the interior of $T$.  See Figure \ref{fig:unitcirclefib}.

\begin{figure}[ht!]
\centerline{
\includegraphics[width=4in]{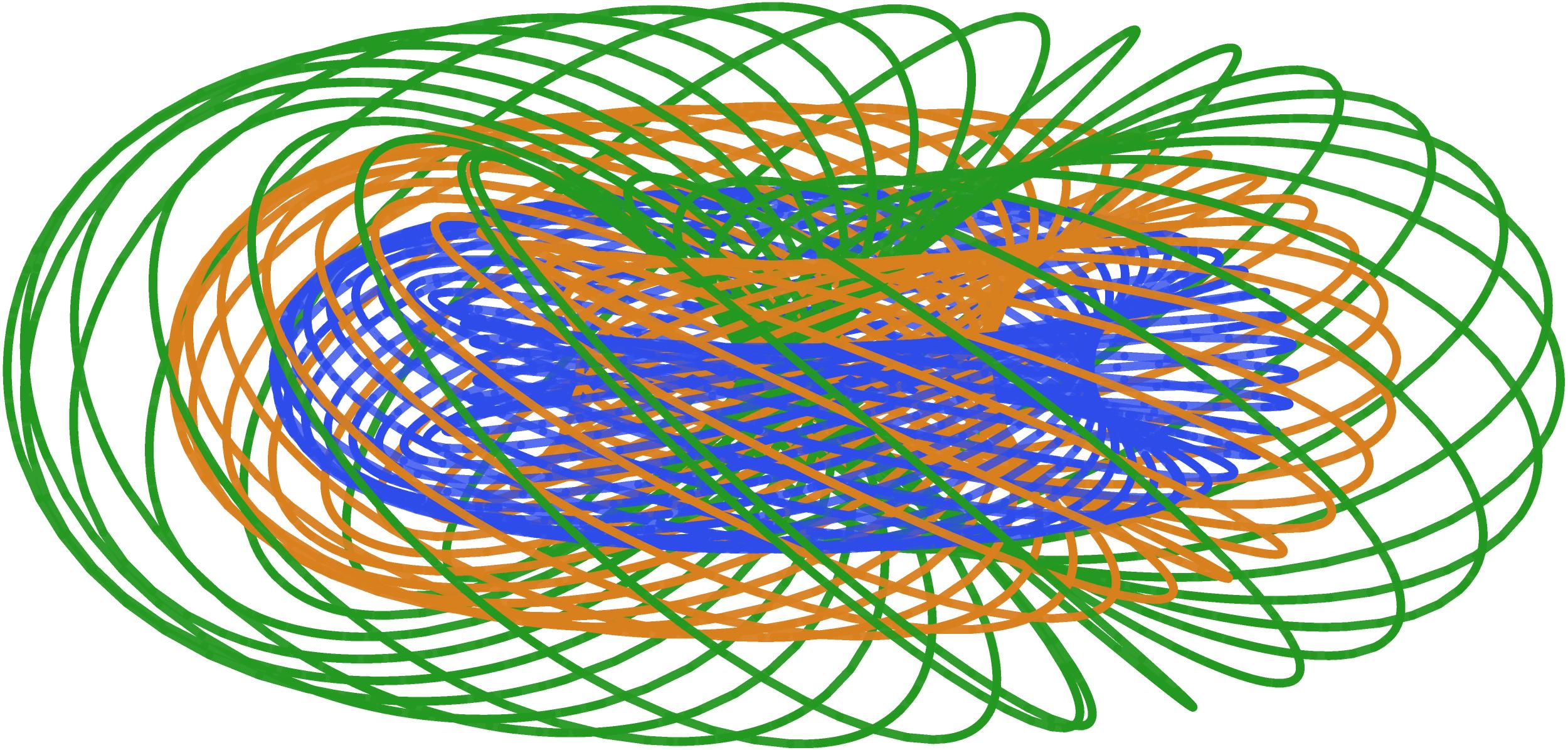}
}
\caption{Fibration of a toroidal region in $\R^3$ by linked unit circles}
\label{fig:unitcirclefib}
\end{figure}

\begin{remark}
We note that the stereographic projection of the Hopf fibration of $S^3$ by unit circles produces an image which looks very similar to that of Figure \ref{fig:unitcirclefib}.  While the circles there are also Villarceau circles on tori, they are not unit circles.
\end{remark}

In Section \ref{sec:ex} we construct explicit unit $n$-sphere fibrations in $\R^{2n+1}$, for $n \in \left\{1, 3, 7 \right\}$.  When $n=1$ this construction produces the fibered torus described above.   The dimensions of our construction are sharp, as shown by our main result:

\begin{theorem}
\label{thm:main}
Suppose that there exists a unit $n$-sphere fibration of an open subset $E \subset \R^d$.  Then $d \geq 2n+1$, and if $d=2n+1$, then $n \in \left\{0, 1, 3, 7 \right\}$.
\end{theorem}

If there exists a unit $n$-sphere fibration of an open region $E \subset \R^d$, then there exists a unit $n$-sphere fibration of $E \times \R \subset \R^{d+1}$, by ``stacking'' unit $n$-sphere fibrations of $E$.  Thus it would be interesting to determine:

\begin{question}
\label{ques:main} Given $n \in \N$, do there exist $d$ and an open subset $E \subset \R^d$ such that $E$ is fibered by unit $n$-spheres?  If so, what is the minimum such dimension $d$?
\end{question}

The answer to this question is known for two other types of geometric fibrations, and it would not be surprising if the answer to Question \ref{ques:main} matches one of these described below:

A \emph{great sphere fibration} is a sphere bundle with total space $S^d$ and fibers which are great $n$-spheres.   The standard example is the Hopf fibration of $S^3$ by unit circles, which arises by choosing an orthogonal complex structure on $\R^4$ and intersecting $S^3$ with all of the complex lines.   Similar constructions with quaternions or octonions yield Hopf fibrations with fibers $S^3$ or $S^7$.  Algebraic topology imposes strong restrictions on the possible dimensions of great sphere fibrations; in particular the fibers must have dimension $0$, $1$, $3$, or $7$.  With this in mind, it would not be surprising if only these $n$ could serve as the fiber dimension of a unit sphere fibration.  See \cite{GluckWarner,GluckWarnerYang,GluckWarnerZiller,McKay} for more information about great sphere fibrations.

A \emph{skew fibration} is a vector bundle with total space $\R^d$ and fibers which are pairwise skew affine $n$-planes.   
In \cite{OvsienkoTabachnikov}, Ovsienko and Tabachnikov showed that a skew fibration of $\R^d$ by $n$-planes exists if and only if $n \leq \rho(d-n) - 1$, where $\rho$ is the Hurwitz--Radon function, defined as follows: Decompose $q \in \N$ as the product of an odd number and $2^{4a+b}$, where $0 \leq b \leq 3$, then $\rho(q) = 2^b + 8a$.  It follows from unboundedness of $\rho$ that for every $n$ there exists $d$ such that $\R^d$ may be fibered by skew affine $n$-planes.  With this in mind, it would not be surprising if every $n$ could serve as the fiber dimension of a unit sphere fibration.   See \cite{Harrison,Harrison2,Harrison4,OvsienkoTabachnikov2} for more information about skew fibrations.

\begin{remark}
\label{rem:skew}
In the proof of Theorem \ref{thm:main}, we will demonstrate a correspondence between linked unit sphere fibrations and skew fibrations.  In particular, we will use the fact that skew fibrations of $\R^{2n+1}$ by $n$-planes exist if and only if $n + 1 = \rho(n+1)$, which occurs if and only if $n \in \left\{ 0, 1, 3, 7 \right\}$.  
\end{remark}

Although we have found some relationships among unit sphere fibrations, great sphere fibrations, and skew fibrations, we have found that the techniques used to study unit sphere fibrations are somewhat different from those used to study the other types of fibrations.  For example, the collection of $n$-spheres in a great sphere fibration of $S^d$ corresponds to a submanifold of $\operatorname{Gr}_{n+1}(d+1)$, and an important component of studying great sphere fibrations is understanding the topology of the Grassmann manifold.  Similarly, the study of skew fibrations requires understanding the topology of the affine Grassmann and spaces of nonsingular bilinear maps.  By contrast, geometry plays a large role in the study of unit sphere fibrations, and it seems unlikely that their study can be completely reduced to questions of topology or linear algebra.  

\section{Proof of the main result}

The proof of the main result uses the notion of linkedness for unit $n$-spheres.  Recall that two disjoint topological $n$-spheres $S_1$ and $S_2$ in $\R^d$ are \emph{linked} if $S_1$ is nontrivial as an element of $\pi_n(\R^d - S_2)$.   Linking of topological $n$-spheres can potentially occur in dimensions $n+2 \leq d \leq 2n+1$,  but linking for unit $n$-spheres is much more restrictive.

\begin{lemma}
\label{lem:linked}
Let $S_1$ and $S_2$ be two disjoint unit $n$-spheres in $\R^d$, $n+2 \leq d \leq 2n+1$, and let $P_i$ be the affine $(n+1)$-plane containing $S_i$.  The spheres $S_1$ and $S_2$ are linked if and only if $S_2$ intersects $P_1$ exactly once in each connected component of $P - S_1$.  In particular, two unit $n$-spheres can only be linked in dimension $d = 2n+1$.
\end{lemma}

\begin{proof}
If $S_2$ does not intersect $P_1$ inside $S_1$, then $S_2$ lives in the complement of the unit disk $D_1$ in $\R^d$ and hence can be contracted to a point inside $\R^d - S_1$.  Similarly, $S_2$ must intersect $P_1$ outside $S_1$. 

The intersection $S_2 \cap P_1$ is a sphere or a single point.   Therefore if $S_2$ intersects $P_1$ in more than two points, at least one in each component, then it intersects $P_1$ in at least a circle, which contains points in each connected component of $P_1 - S_1$ and hence intersects $S_1$.

Similarly, if $d < 2n$, then $P_1$ and $P_2$ must intersect in at least a $2$-dimensional space, and so if $S_2$ intersects $P_1$ in one point in each component,  it must intersect $P_1$ in at least a circle.   Therefore if $S_2$ intersects $P_1$ exactly once in each connected component, then $d= 2n+1$, and it is not difficult to see, by rotating $P_2$, that there is a homotopy bringing $S_2$ to the position of a standard generator for $\pi_n(\R^{2n+1} - S_1)$.
\end{proof}

Now we are prepared to prove the main result.  The proof splits naturally into two separate arguments, which we record as the following two lemmas.

\begin{lemma}
\label{lem:bu}
Suppose that there exists a unit $n$-sphere fibration of an open subset $E \subset \R^d$.  Then $d \geq 2n+1$, and if $d=2n+1$, then the spheres are pairwise linked.
\end{lemma}

The final statement of Lemma \ref{lem:bu}, in the case $n=1$, formalizes the intuitive idea that a $1$-parameter family of unit circles cannot ``pass through'' another unit circle in $\R^3$.

\begin{lemma}
\label{lem:skew}
Suppose that there exists a fibration of an open connected subset $E \subset \R^{2n+1}$ by unit, linked, $n$-spheres. Then $n \in \left\{ 1, 3, 7 \right\}$.
\end{lemma}

We prove the first lemma with a Borsuk--Ulam construction.  For the second lemma we develop and apply a correspondence between unit sphere fibrations and skew fibrations.

\begin{proof}[Proof of Lemma \ref{lem:bu}]
Let $p$ be a point in the interior of $E$.  We use the following notation: $C$ is the unit $n$-sphere fiber through $p$, $B$ is the open $(n+1)$-ball whose boundary is $C$, $P$ is the $(n+1)$-plane containing $C$,  and $D \subset E$ is a small $(d-n)$-dimensional disk centered at $p$, orthogonal to the tangent space to $C$ at $p$, and small enough such that every fiber intersecting $D$ does so transversely and exactly once; see Figure \ref{fig:local}.

\begin{figure}[ht!]
\centerline{
\includegraphics[width=2.5in]{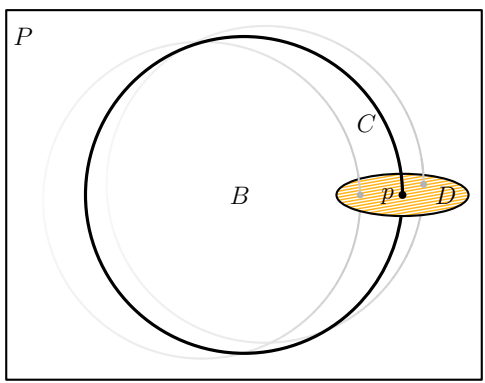}
}
\caption{Fibration by unit circles: local depiction}
\label{fig:local}
\end{figure}

By compactness of $C$ and openness of $E$, there exists $\varepsilon > 0$ such that the $n$-sphere $S \subset P$ concentric with $C$ and of radius $1 + \varepsilon$ lies inside $E$.  Suppose momentarily that for each $x \in S$, the fiber passing through $x$ intersects the boundary $\partial D$ of the small transverse disk at $p$.    This allows us to define $f \colon S \to \partial D \simeq S^{d-n-1}$ as the map which sends $x \in S$ to the intersection of the fiber through $x$ with $\partial D$.

If $d \leq 2n$, i.e.\ if $d-n -1 < n$, then by the Borsuk--Ulam theorem, there exists $x$ such that $f(x) = f(-x)$, and so the fibers through $x$ and $-x$ intersect.  Since $x$ and $-x$ are distance $2 + 2\varepsilon$ apart, their containing fibers are different and hence do not intersect, a contradiction.

Now let $d = 2n+1$.  Since linkedness is preserved by homotopy, there is a dichotomy: all fibers passing through $\partial D$ are linked with $C$, or all fibers are unlinked with $C$.  We assume all fibers are unlinked and derive a contradiction.  Since $C$ is unlinked with the fiber through $x \in S$, and since $x \in S$ lies in the exterior of $C \subset P$, then by Lemma \ref{lem:linked}, the fiber through $x$ does not intersect $B$.  In particular, the image of $f$ misses the point $\partial D \cap B$.   Therefore $f$ may be considered as a map $S^n \to \R^n$, and the Borsuk--Ulam theorem applies in the same way.

Now we address the assumption that for each $x \in S$, the fiber passing through $x$ intersects the boundary $\partial D$ of the small transverse disk at $p$.  This need not occur, and so we modify the definition of $S$ as follows.  Since $D$ is chosen to correspond to a local trivialization, the fibers passing through $D$ generate a foliated neighborhood $U \subset E$ of the fiber $C$, and the fibers passing through $\partial D$ form a fibered topological torus $T = S^{d-n-1} \times S^n$ which serves as the boundary of $U$.  For $q \in C$,  following fibers from $\partial D$ yields a homeomorphism between $\partial D$ and the boundary of a topological $(d-n)$-dimensional disk $D_q$ which is transverse to $C$ at $q$ (explicitly, $D_q$ is the intersection of $U$ with the $(d-n)$-plane orthogonal to the tangent space $T_qC$).   Now let $x$ be the point on the line segment $D_q \cap P$ which is farthest from the center $c$ of $C$.   By compactness of $\partial D_q$, this line segment intersects $\partial D_q$ in finitely many points, and the line segment intersects $\partial D_q$ at least once on each side of $q$, since $C$ lies in the interior connected component of $\R^d - T$.  Therefore $x$ is well-defined, and moreover it varies continuously with $q$.   Said differently, if $\delta_q$ represents the maximum distance from $D_q \cap P$ to $c$, then $x = c + \delta_q q$.  The map $q \mapsto c + \delta_q q = x$ is a homeomorphism from $C$ onto its image $S^\prime$, and the proof goes through with this topological $n$-sphere $S^\prime$ in place of the geometric sphere~$S$.
\end{proof}

\begin{remark} It is tempting to apply Borsuk--Ulam to some natural map on the $n$-sphere $\partial D$, instead of on the topological $n$-sphere $S^\prime$ obtained by fattening the fiber $C$, but we were unable to define a map which simultaneously captures both the unit assumption and the unlinkedness assumption.  It would be interesting to know if the same obstruction could be found by such a map.
\end{remark}

\begin{proof}[Proof of Lemma \ref{lem:skew}]
We show for $n>0$ that the fibration of $E$ induces a fibration of some open set $U \subset \R^{2n+1}$ by skew affine copies of $\R^n$, which is known to only exist when $n \in \left\{1,3,7\right\}$ (see Remark \ref{rem:skew}).

For each spherical fiber $F$, consider the containing $(n+1)$-plane $P$, and then consider the orthogonal $n$-plane $Q$ passing through the center of $F$.  Let $Q_1$ and $Q_2$ be two such $n$-planes obtained from fibers $F_1$ and $F_2$ with corresponding planes $P_1$ and $P_2$.  We claim that if $Q_1$ and $Q_2$ are not skew, then the fibers $F_1$ and $F_2$ are unlinked.

Suppose first that $Q_1$ and $Q_2$ share a common line $\ell$, i.e.\ their linear representatives span fewer than $2n$ dimensions.  Then $\ell$ is orthogonal to both $P_1$ and $P_2$, so $P_1$ and $P_2$ do not span a $(2n+1)$-dimensional space.  If $P_1$ and $P_2$ do not intersect, then $F_1$ and $F_2$ are not linked.  If $P_1$ and $P_2$ intersect, then $F_1$ and $F_2$ both lie in some $2n$-dimensional space, hence are not linked by Lemma \ref{lem:linked}.

Suppose now that the linear representatives of $Q_1$ and $Q_2$ span a $2n$-dimensional space, but that $Q_1$ and $Q_2$ intersect.  Then the vector $u$ connecting the centers of $F_1$ and $F_2$ lies in the linear span of $Q_1$ and $Q_2$.  Moreover, $P_1$ and $P_2$ span a $(2n+1)$-dimensional space and intersect in some line $\ell$, and $\ell$ is orthogonal to $u$.  To verify that $F_1$ and $F_2$ are unlinked, it suffices to understand the intersections of $F_1 \cap P_2$ and $F_2 \cap P_1$.  If either intersection is empty, the spheres are unlinked.  Each nonempty intersection occurs within $P_1 \cap P_2 = \ell$ and must consist of one or two points.

Let $x_i$ be the point on $\ell$ nearest to the center of $F_i$.   Note that $x_i$ can be obtained by taking the intersection of $\ell$ with the copy of the hyperplane orthogonal to $\ell$ passing through the center of $F_i$.  Since $\ell$ is orthogonal to $u$, the centers of $F_1$ and $F_2$ each lie on the same such hyperplane, and therefore $x_1 = x_2$.

Now the intersection of each $F_i$ with $\ell$ is symmetric about $x_1$.  Therefore $F_1$ cannot intersect $P_2$ both outside and inside $F_2$, so the spheres are not linked by Lemma \ref{lem:linked}.

It remains to observe that continuity of the unit sphere fibration leads to continuity of the induced skew covering, and hence yields a bona fide fibration.  Indeed, the skew plane corresponding to a point $p \in E$ is defined by taking the affine plane through the center $c(p)$ and with linear representative $\nu(p) \in \operatorname{Gr}_{n}(2n+1)$, each of which vary continuously with $p$.  Hence the collection of affine planes varies continuously with $p$.
\end{proof}

\section{Example}
\label{sec:ex}

In this section we construct an example of a unit $n$-sphere fibration of $\R^{2n+1}$ for $n \in \left\{ 1, 3, 7 \right\}$.  The construction was inspired by the relationship between unit sphere fibrations and skew fibrations exhibited in the proof of Lemma \ref{lem:skew}.

\begin{example}
We construct an example of a unit sphere fibration for $n = 3$.   The same construction works for $n=1$ and $n=7$ with complex numbers or octonions in place of quaternions.   Write $\R^7 = \R^4 \times \R^3$ and for each $(y,0) \in \R^4 \times \R^3$, let $P_y$ be the  $3$-plane spanned by $(iy,1,0,0)$, $(jy,0,1,0)$, $(ky, 0, 0, 1)$, where $\left\{ i, j, k \right\}$ represents a basis of the imaginary unit quaternions.  Let $Q_y$ be the $4$-plane which passes through $y$ and is orthogonal to $P_y$, and let $S_y$ be the unit $3$-sphere in $Q_y$ centered at $y$.  Now the center map $y \mapsto (y,0) \in \R^7$ and the normal map $y \mapsto Q_y$ are continuous, and therefore the assignment $y \mapsto S_y$ defines a continuous unit sphere fibration provided that no two spheres intersect.

We will show that for distinct $y,z$ with $|y|,|z| < 1$, $S_y$ and $S_z$ do not intersect.

First note that for distinct $y$ and $z$, the linear span of $P_y$ and $P_z$ is $6$-dimensional.  Therefore the affine $4$-planes $Q_y$ and $Q_z$ intersect in precisely a line $\ell$, and it is enough to check that $S_y \cap \ell$ and $S_z \cap \ell$ do not intersect.   Note that $\ell$ contains the origin, and it is easy to check, by taking inner products with the basis vectors of $P_y$ and $P_z$, that $\ell$ has direction
\[
v = (y-z,- \langle iy, z \rangle,  - \langle jy,z \rangle,  - \langle ky, z \rangle).
\]
For future convenience we compute
\begin{align}
\label{eqn:vnorm}
|v|^2 = |y-z|^2 + \langle iy, z \rangle^2 + \langle jy,z \rangle^2 + \langle ky, z \rangle^2 = |y-z|^2 + |z|^2|y|^2 - \langle y,z \rangle^2,
\end{align}
where we have used the Parseval identity together with the fact that $\left\{ y, iy, jy, ky \right\}$ forms an orthogonal basis of $\R^4$.
The nearest point on $\ell$ to the point $y$ is
\[
c_y = \frac{ \langle y,  v \rangle }{|v|^2} v = \frac{\langle y, y-z \rangle}{|v|^2} v,
\]
and the squared distance from $y$ to $c_y$ is
\[
d_y^2 = \langle y - c_y, y - c_y \rangle = |y|^2 - 2\langle y, c_y \rangle + |c_y|^2 = |y|^2 - \frac{\langle y, y-z \rangle^2}{|v|^2}.
\]
Now the intersection points of $S_y$ with $\ell$ are the points $c_y \pm r_y \frac{v}{|v|}$, where $r_y = \sqrt{1 - d_y^2}$.  Let $d$ represent the distance from $c_y$ to $c_z$.   Then
\[
d = |c_y - c_z| = \frac{|y-z|^2}{|v|}.
\]
To check that $S_y$ and $S_z$ do not intersect, we must show that $c_y \pm r_y\frac{v}{|v|} \neq c_z \pm r_z\frac{v}{|v|}$.  For this, it suffices to show that $d < r_y + r_z < \min\left\{ 2r_y + d, 2r_z + d \right\}$; see Figure \ref{fig:ineq}.

\begin{figure}[ht!]
\centerline{
\includegraphics[width=3in]{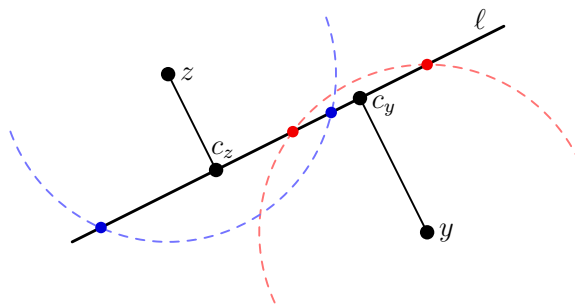}
}
\caption{Two nonintersecting spheres centered at $y$ and $z$; the red points represent $c_y \pm r_y\frac{v}{|v|}$, and the blue points represent $c_z \pm r_z\frac{v}{|v|}$}
\label{fig:ineq}
\end{figure}

\begin{lemma}
In the notation above, we have
\begin{compactenum}[(a)]
\item $d < r_y + r_z$,
\item $r_z < r_y + d$,
\item $r_y < r_z + d$.
\end{compactenum}
\end{lemma}

\begin{proof}
The inequality in part (a) can be written explicitly
\[
\frac{|y-z|^2}{|v|} < \sqrt{1 - |y|^2 + \frac{\langle y, y-z \rangle^2}{|v|^2}} + \sqrt{1 - |z|^2 + \frac{\langle z, y-z \rangle^2}{|v|^2}},
\]
or
\begin{align}
\label{eqn:lema}
|y-z|^2 < \sqrt{|v|^2 - |v|^2|y^2| + \langle y, y - z \rangle^2} + \sqrt{|v|^2 - |v|^2|z^2| + \langle z, y - z \rangle^2}.
\end{align}
Since $|y| < 1$, we have
\[
0 < |v|^2-|v|^2|y|^2,
\]
therefore
\[
\langle y, y-z \rangle^2 < |v|^2 - |v|^2|y^2| + \langle y, y - z \rangle^2,
\]
and so
\begin{align}
\label{eqn:yin}
\langle y, y-z \rangle < \sqrt{|v|^2 - |v|^2|y^2| + \langle y, y - z \rangle^2}.
\end{align}
Similarly,
\begin{align}
\label{eqn:zin}
-\langle z, y-z \rangle < \sqrt{|v|^2 - |v|^2|z^2| + \langle z, y - z \rangle^2}.
\end{align}
Now summing (\ref{eqn:yin}) and (\ref{eqn:zin}) yields (\ref{eqn:lema}).

The inequality in part (b) can be written explicitly 
\[
\sqrt{1 - |z|^2 + \frac{\langle z, y-z \rangle^2}{|v|^2}} < \sqrt{1 - |y|^2 + \frac{\langle y, y-z \rangle^2}{|v|^2}} + \frac{|y-z|^2}{|v|}.
\]
Multiplying by $|v|$ and using (\ref{eqn:vnorm}) yields
\[
\sqrt{|y-z|^2 - |z|^2(|y|^2|z|^2 - \langle y, z \rangle^2)} < \sqrt{|y-z|^2 - |y|^2(|y|^2|z|^2 - \langle y, z \rangle^2)} + |y-z|^2.
\]
Squaring both sides and rearranging, we obtain the inequality
\begin{align}
\label{eqn:lemb}
(|y|^2-|z|^2)(|y|^2|z|^2 - \langle y, z \rangle^2) < 2\sqrt{|y-z|^2 - |y|^2(|y|^2|z|^2 - \langle y, z \rangle^2)}|y-z|^2 + |y-z|^4.
\end{align}
For convenience let $a = |y|$, $b=|z|$ and $c = \cos \theta$, where $\theta$ is the angle between $y$ and $z$.  We write
\begin{align*}
0 & \geq -(a-bc)^2-(ac-b)^2 \\ & = -a^2-b^2-a^2c^2-b^2c^2+4abc \\
& = (a^2+b^2)(1-c^2) -2(a^2-2abc+b^2)\\ & \geq (a^3b^2+a^2b^3)(1-c^2)-2(a^2-2abc+b^2), \hspace{.2in} \mbox{ since } a,b<1.
\end{align*}
Now we may assume $a > b$, since otherwise (\ref{eqn:lemb}) is immediate.  Then multiplying the previous inequality by $a-b$ yields:
\begin{align*}
0 & \geq (a-b)(a^3b^2+a^2b^3)(1-c^2)-2(a-b)(a^2-2abc+b^2) \\
& = (a^4b^2-a^2b^4)(1-c^2) - 2(a-b)(a^2-2abc+b^2).
\end{align*}
Therefore
\begin{align*}
(a^4b^2-a^2b^4)(1-c^2)& \leq 2(a-b)(a^2-2abc+b^2) \\ &  = 2(a-b)|y-z|^2 \\ & = 2\sqrt{a^2-2ab+b^2}|y-z|^2 \\ & \leq 2\sqrt{a^2-ab(1+2c-c^2)+b^2}|y-z|^2, \hspace{.5in} \mbox{since } -1 \leq c \leq 1,
\\ & = 2\sqrt{a^2 - 2abc + b^2 -ab(1-c^2)}|y-z|^2
\\ & \leq 2\sqrt{a^2 - 2abc + b^2 -a^4b^2(1-c^2)}|y-z|^2
\\ & = 2\sqrt{|y-z|^2 -a^4b^2(1-c^2)}|y-z|^2
\\ & < 2\sqrt{|y-z|^2 -a^4b^2(1-c^2)}|y-z|^2 + |y-z|^4,
\end{align*}
and this is (\ref{eqn:lemb}).

Part (c) is similar.
\end{proof}


\end{example}

\section{History}

Questions regarding the possibility of fibering an open set in $\R^3$ by unit circles were raised by the first named author (DA) as a graduate student at UC Berkeley circa 1970.  No one questioned had answers, but in the early 1990s DA discovered the fibering of the toroidal region depicted in Figure \ref{fig:unitcirclefib} (he called this region a ``bialy'' in reference to the bagel-like cuisine).  He gave a series of talks on his findings, which were published in a 1993 technical report ``Hoops in $\R^3$,'' though he later retracted this report after discovering an error.

In 2019 the other three authors (OA) were introduced to this type of problem at Carnegie Mellon, when colleague Anton Bernshteyn posed to them the existence question for Borel coverings of $\R^3$ by unit circles.   OA became interested in the continuous case, and a literature search resulted in only a single finding: a 1995 blog post by Evelyn Sander, which she wrote to summarize a 1994 talk by DA at UIUC.   Her post mentioned the announced results of DA, but OA were unable to extract any formal proofs, though through this post they did learn the importance of linkedness and the potential relationship to division algebras.  They were unable to access the technical report nor find any contact information for DA. 

OA later discovered that a mutual colleague had contact information for DA.  OA shared a draft of this article and learned that DA also knew (though had not published) proofs of some of these results.  This led to the present collaboration, which we hope provides some closure to a 50 year old question -- though of course many interesting questions remain!

\bibliographystyle{plain}
\bibliography{circlesbib}{}

\end{document}